\documentclass[11pt, reqno]{amsart}
\usepackage{amssymb}
\usepackage{amscd}
\DeclareMathAlphabet{\mathsl}{OT1}{cmr}{m}{sl}

\numberwithin{equation}{section}

\newcommand{\Z}{\mathbb{Z}}

\newcommand{\C}{\mathbb{C}}

\newcommand{\grF}{\mathsf{F}}

\newcommand{\grE}{\mathsf{E}}

\newcommand{\grB}{\mathsf{B}}
\newcommand{\grU}{\mathsf{U}}
\newcommand{\grZ}{\mathsf{Z}}
\newcommand{\grM}{\mathsf{M}}

\newcommand{\id}{\operatorname{\mathsl{id}}}

\renewcommand{\max}{\operatornamewithlimits{\mathsl{max}}}
\renewcommand{\deg}{\operatorname{\mathsl{deg}}}
\renewcommand{\dim}{\operatorname{\mathsl{dim}}}
\renewcommand{\ker}{\operatorname{\mathsl{ker}}}
\newcommand{\ann}{\operatorname{\mathsl{ann}}}
\newcommand{\cor}{\operatorname{\mathsl{cor}}}

\newcommand{\grC}{\mathsf{C}}

\renewcommand{\int}{\operatorname{\mathsl{int}}}

\DeclareMathOperator{\gr}{\operatorname{\mathsf{gr}}}
\DeclareMathOperator{\Nrd}{\operatorname{\mathsl{Nrd}}}
\DeclareMathOperator{\DNrd}{\operatorname{\mathsl{DNrd}}}
\DeclareMathOperator{\ind}{\operatorname{\mathsl{ind}}}
\DeclareMathOperator{\Aut}{\operatorname{\mathsl{Aut}}}
\DeclareMathOperator{\im}{\operatorname{\mathsl{im}}}
\DeclareMathOperator{\Gal}{\operatorname{\mathcal{G}}}
\DeclareMathOperator{\Br}{\operatorname{\mathsl{Br}}}
\DeclareMathOperator{\lcm}{\operatorname{\mathsl{lcm}}}

\DeclareMathOperator{\charac}{\operatorname{\mathsl{char}}}

\def\hsp{\mspace{1mu}}
\newcommand{\DIM}[2]{[#1{\hsp:\hsp}#2]}
\newcommand{\IND}[2]{\lvert#1{\hsp:\hsp}#2\rvert}

\def\tsum{\textstyle\sum\limits}

\newcommand{\jdotfont}{}
\font\jdotfont lcircle10  scaled 913 
\newcommand{\osbullet}{\jdotfont\char113} 
\newlength{\sbwd} 
\newcommand{\csbullet}{\kern.5\sbwd\osbullet\kern-.5\sbwd}

\def\SK{\mathrm{SK}_1}
\def\ga{\gamma}
\def\Ga{\Gamma}
\def\inv{^{-1}}
\def\Div{\operatorname{\mathsl{Div}}}
\def\DT{\Div(T)}
\def\tDT{\widetilde{\DT}}
\def\DTt{\Div(T)^\tau}
\def\SSS{\mathcal S}
\def\dd{\delta}
\def\OM{\Omega}
\def\modd{\operatorname{\mathsl{mod}}}
\def\tNG{\widetilde{N_G}}
\def\OO{\mathcal O}
\def\VV{\mathcal V}
\newcommand{\St}[1]{\Sigma_\tau(#1)}
\newcommand{\Stp}[1]{\Sigma_\tau'(#1)}
\def\wi{\widetilde}

\newtheorem{proposition}{Proposition}[section]
\newtheorem{theorem}[proposition]{Theorem}
\newtheorem{corollary}[proposition]{Corollary}
\newtheorem{lemma}[proposition]{Lemma}

\theoremstyle{definition}

\theoremstyle{remark}
\newtheorem*{remark}{Remark}

\newtheorem{example}[proposition]{Example}

\begin{document}
\title{Unitary $SK_1$ for a graded division ring and its quotient division ring}
\author{A. R. Wadsworth \and V. I. Yanchevski\u\i}
\address{Department of Mathematics\\
University of California, San Diego\\
9500 Gilman Drive \\
La Jolla, CA  92093-0112\\
USA}
\email{arwadsworth@ucsd.edu}
\address{Institute of Mathematics of the National Academy of Sciences of Belarus\\
Ul. Surganova  11\\
220072 Minsk\\
Belarus}
\email{yanch@im.bas-net.by}

\maketitle
\begin{abstract}
Let $\grE$ be a graded division ring finite-dimensional over its center
with torsion-free abelian grade group, and let $q(\grE)$ be its 
quotient division ring.  Let $\tau $ be a degree-preserving unitary involution
on $\grE$.  We prove that $\SK( \grE, \tau) \cong \SK( q(\grE), \tau)$.
\end{abstract}
\section{Introduction}

Let $D$ be a division algebra finite-dimensional over its center $K$. Let $\tau$
be a unitary involution on $D$, i.e., $\tau$ is an antiautomorphism of $D$
with $\tau^2= \id$ such that $\tau|_K \ne \id$.  By definition,
$$
\SK(D, \tau) \, = \, \Sigma_\tau'(D) \big/ \Sigma_\tau(D),
$$
where
$$
\Sigma_\tau(D) \, = \big\langle\{a\in D^*\mid   a = \tau(a) \}\big\rangle
\qquad \text{and} \qquad 
\Sigma_\tau'(D) \, = \,\{ a\in D^* \mid \Nrd_D(a) = \tau(\Nrd_D(a))\}.
$$

Analogously, suppose $\grE = \bigoplus_{\gamma\in \Ga_\grE}\grE_\ga$
 is a graded division algebra (i.e., a graded ring in which all nonzero 
homogeneous elements are units) of finite rank over its center $\grZ$, and with 
torsion-free abelian grade group $\Ga_\grE$.  Suppose $\grE$ has a unitary graded 
involution $\tau$, which is a ring antiautomorphism with $\tau(\grE_\ga) = \grE_\ga$
for each $\ga$,  such that $\tau^2 = \id$ and $\tau|_\grZ \ne \id$.  
Then there is a reduced norm map $\Nrd_\grE\colon \grE \to \grZ$, and one can 
define $\St\grE$, $\Stp \grE$, and $\SK(\grE, \tau)$ just as we did above for 
$D$.  Any such $\grE$ has a ring of (central) quotients, 
$q(\grE) = \grE \otimes _\grZ q(\grZ)$, where $q(\grZ)$ is the quotient
field of the integral domain $\grZ$, and it is known that 
$q(\grE)$ is a division ring finite-dimensional over its center, which 
is $q(\grZ)$. The unitary graded involution $\tau$ on $\grE$ extends canonically 
to a unitary involution on $q(\grE)$, also called $\tau$.  In this paper
we will prove the following:

\begin{theorem}\label{th:main}
Let $\grE$ be a graded division algebra $($with torsion-free abelian grade group$)$ 
finite-dimensional over its center, and let $\tau$ be a unitary graded  
involution on $\grE$.
Let $q(\grE)$ be the quotient division ring of~$\grE$ and let $\tau$ denote
also  
the extension of $\tau$ to a unitary involution of $q(\grE)$.  Then,
$$
\SK(\grE,\tau) \,\cong\, \SK(q(\grE), \tau).
$$  
\end{theorem}

This theorem complements a result in \cite{hazw2}:  Suppose $D$ is a division
algebra finite-dimensional over its center $K = Z(D)$, and suppose $K$ has a Henselian 
valuation $v$.  It is well-known that $v$ has a unique extension to a valuation
on $D$.  The filtration on $D$ induced by the valuation yields an associated graded 
ring $\gr(D)$ which is a graded division algebra of finite rank over its center.  
If $\tau$ is a unitary involution on $D$ which is compatible with the valuation, 
there is an induced   graded  
involution $\wi \tau$ on~$\gr(D)$.  It was shown in \cite[Th.~3.5]{hazw2} that if the restriction
of $v$ to the $\tau$-fixed field $K^\tau$ is Henselian and $K$ is 
tamely
\vfill\eject
\noindent ramified over
$K^\tau$,  
 then $\tau$ is compatible with 
$v$, $\widetilde \tau$ is unitary,  and  
\begin{equation}\label{assocgraded}   
\SK(D, \tau) \, \cong \, \SK(\gr(D), \wi\tau).
\end{equation}
This theorem served to focus attention on unitary $\SK$ for graded division algebras.
The graded division algebra $\gr(D)$ has a significantly simpler structure than 
the valued division algebra $D$; notably $\gr(D)$~has a much more tractable group of units 
(as they are all homogeneous).  Consequently, $\SK$ calculations are often
 substantially easier
 for $\gr(D)$ than for $D$.  This was demonstrated 
in \cite{hazw2} and  \cite{W},  where many formulas for $\SK(\gr(D), \wi\tau)$
were proved.  By that approach new results  on $\SK(D, \tau)$ were  
obtained, as well as new and simpler proofs of results that had previously been 
proved by calculations that were often quite complicated.
 Now, by virtue of the theorem 
proved here, all the results proved in \cite{hazw2} and \cite{W} on $\SK(\grE,\tau)$
for a graded division algebra $\grE$ carry over to yield corresponding results for 
$\SK(q(\grE), \tau)$.

One of the key results for unitary $\SK$ is the \lq \lq Stability Theorem,"  
which says that the unitary $\SK$ is unchanged on passage from a division algebra
$D$ to a  rational division algebra over $D$.  (This was originally proved in 
\cite[\S 23]{y1var}.)  We will show at the end of \S\,4 that the Stability Theorem
is a quick  corollary of our  theorem. 

When the torsion-free abelian grade group $\Ga_\grE$ of a graded division algebra
$\grE$ is finitely-generated (hence a free abelian group), $\grE$ has a concrete 
description as an iterated twisted Laurent polynomial ring over the 
division ring $\grE_0$, as follows: Take any homogeneous 
$x_1, \ldots, x_n$ in the group $\grE^*$ of units of $\grE$ such that 
$\Ga_\grE= \Z\deg(x_1) \oplus \ldots \oplus \Z\deg(x_n)$.  Then
$\grE = \grE_0[x_1, x_1\inv, \ldots, x_n, x_n\inv; \sigma_1, \ldots, \sigma_n]$, 
i.e., $\grE=\grE_n$, where for $i = 1, 2, \ldots, n$, $\grE_i = \grE_{i-1}
[x_i, x_i\inv; \sigma_i]$, which is a Laurent polynomial ring with multiplication twisted
by the relation $x_i c = \sigma_i(c) x_i$ for all $c\in \grE_{i-1}$.  Each $\sigma_i$
is a graded (i.e., degree-preserving) automorphism of~$\grE_{i-1}$, and $\sigma_i$  is 
completely determined by its action on $\grE_0$ and on $x_1, x_2, \ldots, x_{i-1}$.
To assure that $\grE$~is finite-dimensional over its center, it is assumed that 
some power of $\sigma_i$ is an inner automorphism of $\grE_{i-1}$.  
The quotient division ring $q(\grE)$ is the iterated twisted rational division algebra
\begin{equation} \label{An}
A_n = \grE_0(x_1, \ldots, x_n; \sigma_1, \ldots, \sigma_n),
\end{equation}
 which is also
the quotient division algebra of the iterated twisted polynomial ring 
$\grE_0[x_1, \ldots, x_n; \sigma_1, \ldots, \sigma_n]$.  
 
 In \cite{ynvar} the second author gave formulas for $\SK(A_n, \tau_n)$
for the iterated twisted rational division algebras $A_n$ as in \eqref{An}
above, for a unitary involution $\tau_n$ on $A_n$ arising from a unitary graded
involution on the iterated twisted polynomial ring $\grE_0[x_1, \ldots, x_n; 
\sigma_1, \ldots, \sigma_n]$.  The crucial identity
\begin{equation}\label{SigAn}
\Sigma_{\tau_n}'(A_n) \, =\, \big(\Sigma_{\tau_n}'(A_n) \cap \grE_0\big) \cdot
\Sigma_{\tau_n}(A_n).
\end{equation}
is given in that paper, with a much
too brief and incomplete sketch of proof.
(More detailed proofs of this identity and other results in \cite{ynvar} are given in 
this author's habilitation, which, regrettably, was never published.
A complete proof of the $n = 1$ case was given in \cite{y1var}.) As will be seen
in \S\,4 below, the proof of this identity, stated there as \eqref{claim}, constitutes the 
bulk of the proof of our Theorem.   
Our proof of this identity is by induction on $n$, as is done in \cite{ynvar}, but the
proof differs susbstantially, in that it is carried out as much as possible 
 by calculations 
in the divisor group~$\DT$ of the twisted polynomial ring in one variable
$T = \grE_0[x_1;\sigma_1]$.  This $\DT$ is a free abelian group, and calculations 
with it  
are significantly simpler  than with the quotient division ring~$q(T)$.   
The style of proof here is analogous to the approach in 
\cite[\S\,5]{hazw1} to proving the corresponding result for the
nonunitary $\SK(\grE)$ of a graded division algebra. 

Th.~\ref{th:main} above, when combined with \eqref{assocgraded},  provides a unifying
perspective for understanding why there is such a great similarity 
between the formulas for 
$\SK(D,\tau)$ for $D$ a division algebra over a Henselian field, as in 
\cite{y} and \cite{yy} and the formulas for $\SK(\grE_0(x_1, \ldots, x_n; \sigma_1,
\ldots, \sigma_n), \tau_n)$ given in \cite{y1var} and\cite{ynvar}:    The formulas in each 
case coincide with   
formulas for $\SK(\grE, \tau)$ for a related graded division algebra $\grE$. 

 Theorem~\ref{th:main} and formula \eqref{assocgraded} are unitary analogues to results for nonunitary 
$\SK$ for division algebras over Henselian fields,  graded division algebras, and 
their associated quotient division algebras given in \cite[Th.~4.8, Th.~5.7]{hazw1}.   
This is another manifestation of the philosophy that results about $\SK$ ought
to have corresponding results for the unitary $\SK$.  From the perspective of
algebraic groups, this pholosophy is motivated by the fact that division algebras
are associated with algebraic groups of inner type $A_n$, while division algebras
with unitary involution are associated algebraic groups of outer type~$A_n$
(see, e.g., \cite[Ch.~VI]{BoI}).

\bigskip 

\section{Graded division algebras and unitary involutions}

In this section we recall some basic known facts about graded division algebras
and unitary involutions which will be used in the proof of Th.~\ref{th:main}.
A good reference for the properties of graded division algebras stated here 
without proof is \cite{HwW2}.

Let $\Gamma$ be a torsion-free abelian group.  A ring $\grE$ is a {\it graded 
division ring} (with grade group in $\Gamma$) if $\grE$ has additive subgroups 
$\grE_\gamma$ for $\gamma \in \Gamma$ such that 
$\grE = \bigoplus_{\gamma\in \Gamma}\grE_\gamma$ 
and $\grE_ \gamma \cdot \grE_\delta =
\grE_{\gamma+\delta}$ for all $\gamma, \delta \in \Gamma$, and each 
$E_{\gamma}\setminus \{0\}$ lies in~$\grE^*$, the group of units of $\grE$. 
For background on graded division rings and proofs  of their properties mentioned
here, see \cite{HwW2}.
The grade 
group of $\grE$  is 
$$
\Gamma_\grE \, = \, \{\gamma\in \Gamma \mid \grE_\gamma \ne \{0\}\,\},
$$
a subgroup of $\Gamma$.  For $a\in \grE_\gamma \setminus \{0\}$ we write
$\deg(a) = \gamma$.  A significant property is that 
$\grE^* = \bigcup_{\ga \in \Ga_\grE} E_\ga\setminus\{0\}$, i.e., 
every unit of $\grE$ is actually homogeneous.  ($\Ga_\grE$ torsion-free is 
needed for this.)  Thus, $\grE$ is not a division ring if $|\Ga_\grE|>1$.
But, $\grE$ has no zero divisors.  (This also depends on having $\Ga_\grE$ 
torsion-free.)  However, $\grE_0$ is a division ring, and each 
$\grE_\ga$ ($\ga\in \Ga_\grE$) is a $1$-dimensional left- and right-
$\grE_0$-vector  space.  

Let $\grM$ be any graded left $\grE$-module, 
i.e., $\grM$ is a left $\grM$-module with additive subgroups $\grM_\ga$ 
such that $\grM = \bigoplus_{\ga \in \Ga} \grM_\ga$ and 
$\grE_\ga\cdot \grM_\varepsilon \subseteq \grM_{\ga+\varepsilon}$ for all $\gamma,
\varepsilon\in \Ga$. Then, $\grM$ is a free $\grE$-module with 
homogeneous base, and any two such bases have the same cardinality, 
which is called the dimension, $\dim_\grE(\grM)$;  $\grM$ is therefore said 
to be a graded vector space over $\grE$.  

Let $\grZ= Z(\grE)$, the center of $\grE$, which is a graded subring of $\grE$.
Indeed, $\grZ$ is a {\it graded field}, i.e., a commutative graded division ring.
Then $\grE$ is a  left (and right) graded $\grZ$-vector space, and we 
write $\DIM\grE\grZ$ for $\dim_\grZ(\grE)$.  In this paper we work exclusively 
with finite-dimensional graded division algebras, i.e., those $\grE$ with 
$\DIM \grE{\grZ}<\infty$.  Clearly $\grZ_0$ is a field, and $\grE_0$
is a finite-dimensional $\grZ_0$-algebra.  Moreover, $\Ga_\grZ$ is a 
subgroup of $\Ga_\grE$, and it is easy to verify the \lq\lq fundamental
equality"
\begin{equation}\label{FE}
\DIM \grE\grZ \, = \, \DIM{\grE_0}{\grZ_0}\, \IND{\Ga_\grE}{\Ga_\grZ}.
\end{equation}
We have $\grZ_0\subseteq Z(\grE_0)\subseteq \grE_0$.  Let $\Gal(Z(\grE_0)/\grZ_0)$
be the Galois group for the finite-degree field extension $Z(\grE_0)$ of
$\grZ_0$.  There is a well-defined canonical map 
$$
\theta_\grE\colon \Ga_\grE \to \Gal(Z(\grE_0)/\grZ_0) \qquad \text{given by}
\qquad\theta_\grE(\deg(a)) \colon c\mapsto a\hsp c\hsp a\inv \ \ 
\text{for all} \ \  a\in \grE^*, c\in \grE_0. 
$$
clearly $\Ga_\grZ \subseteq \ker(\theta_\grE)$, so $|\im(\theta_E)|\le 
\IND{\Ga_\grE}{\Ga_\grZ} <\infty$.  Moreover, the fixed field of 
$\im(\theta_\grE)$ is $\grZ_0$.  Hence, $Z(\grE_0)$ is Galois over 
$\grZ_0$ with abelian Galois group $\Gal(Z(\grE_0)/\grZ_0) = \im(\theta_\grE)$.

  Since $\grZ$ is a commutative ring with no zero divisors, it has a quotient field 
$q(\grZ)$.  Then $\grE$ has its ring of central quotients 
$$
q(\grE) \, = \, \grE \otimes_\grZ q(\grZ).
$$
Because $\grE$ is a free, hence torsion-free $\grZ$-module, the canonical map 
$\grE \to q(\grE)$, $a\mapsto a\otimes 1$, is injective.  Therefore, we 
view $\grE$ as a subring of $q(\grE)$.
Note that $q(\grE)$ has no zero divisors since $\grE$ has none.  
Furthermore,  $q(\grE)$ is a 
$q(\grZ)$-algebra with ${\DIM {q(\grE)}{q(\grZ)} = \DIM \grE \grZ <\infty}$.  
Hence, $q(\grE)$ is a division ring, called the quotient division algebra of
$\grE$.  Clearly, $Z(q(\grE)) = q(\grZ)$. The index of $\grE$ is defined to be
${\ind(\grE) = \sqrt{\DIM \grE\grZ} = \ind(q(\grE)) \in \Z}$.  

It is known that $\grE$ is an Azumaya algebra over $\grZ$, and hence from general 
principles that there is a reduced norm map $\Nrd_\grE\colon \grE \to \grZ$.  
In fact, by \cite[Prop.~3.2(i)]{hazw1}, $\Nrd_\grE$ coincides with the restriction 
to~$\grE$ 
of the usual reduced norm $\Nrd_{q(E)}$ on $q(\grE)$.  Also, by 
\cite[Prop.~3.2(iv)]{hazw1}, for $a\in \grE_0$, 
we have
\begin{equation}\label{normformula}
\Nrd_\grE(a) = N_{Z(\grE_0)/\grZ_0}(\Nrd_{\grE_0}(a))^\lambda\qquad\text{where}\quad
\lambda = \ind(\grE)\big/ \big(\ind(\grE_0) \cdot \DIM{Z(\grE_0}{\grZ_0}\big),
\end{equation}
where $\Nrd_{\grE_0}$ is the reduced norm for $\grE_0$ and $N_{Z(\grE_0)/\grZ_0}$
is the field norm from $Z(\grE_0)$ to $\grZ_0$.

A graded involution on the graded division algebra $\grE$ is a ring antiautomorphism
$\tau\colon \grE \to \grE$ such that $\tau^2 = \id_\grE$ and $\tau(\grE_\gamma)
 = \grE_\gamma$ for each $\gamma \in \Gamma$.    Such a $\tau$ is said to be 
{\it unitary} (or of the second kind) if $\tau|_\grZ \ne \id_\grZ$ where 
$\grZ = Z(\grE)$. Assuming $\tau$ is unitary, let 
$\grF = \grZ^\tau = \{ c\in \grZ\mid \tau(c) = c\}$, which is a graded subfield
of $\grZ$ with $\DIM \grZ \grF = 2$. It follows from the fundamental equality
that  either $\DIM{\grZ_0}{\grF_0} = 2$ and $\Gamma_\grZ = \Gamma_\grF$ or 
$\grZ_0 =\grF_0$ and $\IND{\Gamma_\grZ}{\Gamma_\grF} = 2$.  The second case, 
where the involution induced by $\tau$ on $\grE_0$ is not unitary, tends to be 
uninteresting, but it can occur.  
We write $\tau$ also for the induced involution 
$\tau\otimes \id_{q(\grF)}$  on $q(\grE) = \grE \otimes_\grF q(\grF)$.  
Recall  that for all $a\in \grE$ we have 
\begin{equation}\label{Nrdcompat}
\Nrd_\grE(\tau(a))\, = \, \tau(\Nrd_E(a)),
\end{equation}
since this equality holds for $\Nrd_{q(\grE)}$. 
(For, if $a\in q(\grE)$, then $\Nrd_{q(\grE)}(a)$
is determined by the minimal polynomial $p_a$ of $a$ over 
$Z(q(\grE))$, and  $p_{\tau(a)} = \tau(p_a)$.) 
 
If $\tau'$ is another unitary graded involution on $\grE$, we write
$\tau \sim \tau' $ if $\tau|_\grZ = \tau'|_\grZ$.  In particular, for any 
$c\in \grE^*$, if $\tau(c) c\inv \in \grZ^*$, then $\tau' = \int(c) \circ \tau$
is a unitary graded involution on $\grE$, and $\tau' \sim \tau$.  
Here $\int(c)$ denotes the inner automorphism $a\mapsto cac\inv$ of $\grE$.
Since $c$ is homogeneous, $\int(c)$ is clearly a graded (i.e., 
degree-preserving) automorphism of $\grE$.  

For a unitary graded involution $\tau$ on $\grE$, set 
$
S_\tau(\grE) \, = \, \{ a\in \grE^* \mid \tau(a) = a\}, 
$
the set of symmetric units of $\grE$, and set
$$
\Sigma_\tau(\grE)\, = \, \langle S_\tau(\grE)\rangle \qquad \text{and}
\qquad \Sigma_\tau'(\grE) \,=\, \{ a\in \grE^* \mid \Nrd_
\grE(a) \in S_\tau(\grE)\}.
$$
Then, by definition, $\SK(\grE, \tau) = \Stp \grE/\St \grE$.  


We recall a few fundamental facts about unitary involutions 
on ungraded division algebras which have analogues for graded division aglebras:

\vfill\eject 

\begin{lemma}\label{lem:equivrem} 
Let $D$ be a division algebra finite-dimensional over its center, and let 
$\tau$ and $\tau'$ be unitary involutions on $D$.
\begin{enumerate}
\item[(a)]
If  $\tau\sim \tau'$
$($i.e., $\tau|_{Z(D)} = \tau'|_{Z(D)}$$)$, then 
${\Sigma_{\tau'}'(D) = \Sigma_\tau'(D)}$ and
${\Sigma_{\tau'}(D) = \St D}$,\break so ${\SK(D,\tau') = \SK(D, \tau)}$. 
\item[(b)]
$[D^*,D^*] \subseteq \St D$, where $[D^*, D^*] = 
\langle aba\inv b\inv\mid a,b\in D^*\rangle$.
\end{enumerate}
\end{lemma}

For a proof of (a), see \cite[Lemma~1]{yin}, and for (b) see 
\cite[Prop.~17.26, p.~267]{BoI}.  Part (b) was originally proved by 
Platonov and Yanchevski\u\i.  See \cite[Remark~4.1(ii), Lemma~2.3(iv)]{hazw2} 
for the graded versions of (a) and (b).


\bigskip

\section{The divisor group of a twisted polynomial ring}
The proof of Th.~\ref{th:main}, both in the case  $\Gamma_\grE \cong \Z$ and also 
in the induction argument for $\Gamma_\grE \cong \Z^n$ will use properties 
of twisted polynomial rings in one variable over a division ring.
In this section we give the properties we need about the divisor group
of such a twisted polynomial ring.

Let $D$ be a division ring finite dimensional over its center $K$, and let $\sigma $
be an automorphism of $D$ whose restriction to $K$ has finite order, and let 
$T$ be the twisted polynomial ring
$$
T \, =\, D[x;\sigma], 
$$
consisting of polynomials $\sum_{i = 0}^k a_ix^i$ with $a_i \in D$, with the usual addition
of polynomials, but multiplication twisted by $\sigma$, so that 
$$
(ax^i)(bx^j)\, = \, a\sigma^i(b)x^{i+j}\qquad \text{for all}\ a,b\in D, \ i,j \ge0.
$$
For the factorization theory of such rings $T$, \cite[Ch.~1]{J} is an excellent 
reference.  Let $A = q(T) = D(x;\sigma)$, the quotient division ring of $T$, 
which is a twisted rational function division algebra in one variable.  We will make 
fundamental use of the \lq\lq divisor group" $\DT$ described in \cite[\S\,5]{hazw1}:    
Let $\mathcal S$ be the set of isomorphism classes $[S]$ of simple left $T$-modules
$S$; then 
$$
\DT\,=\, \textstyle \bigoplus\limits_{[S] \in \mathcal S} \Z[S],
$$ 
the free abelian group on $\SSS$.  (Note that in the commutative case when $D$ is a field 
and $\sigma = \id$, then $T$~is a polynomial ring, and $\DT$ is its usual 
divisor group.  This is the source of the terminology.) 
For simple $T$-modules $S, S'$, we have $\ann_T(S)$ and $\ann_T(S')$ are maximal
two-sided ideals of $T$, and $S\cong S'$ iff $\ann_T(S) = \ann_T(S')$.  Thus, we could 
have indexed $\DT$ by the maximal two-sided ideals of $T$; but, the indexing by 
simple modules is more natural for our purposes.
 We call an element 
$\alpha = \sum n_{[S]} [S]$ of $\DT$ a {\it divisor}, and call it an {\it effective divisor}
if every $n_{[S]} \ge 0$.   

Also, there is a degree homomorphism 
$$
\deg\colon \DT \to \Z \qquad \text{given by} \quad\deg\big(\tsum n_{[S]} [S]\big) \, = \,
\tsum n_{[S]} \dim_D(S).
$$
Note that if $M$ is any left $T$-module of finite length
(equivalently, finite-dimensional as a $D$-vector space), then $M$ determines
an effective Jordan-H\" older divisor
$$
jh(M) \, = \tsum \limits_{i = 1}^k [M_i/M_{i-1}],
$$
where $\{0\} = M_0\subsetneqq M_1 \subsetneqq \ldots \subsetneqq M_k= M$
is a chain of $T$-submodules of $M$ with each $M_i/M_{i-1}$ simple.  
The Jordan-H\" older Theorem shows that $jh(M)$ is well-defined.

Every simple left $T$-module has the form $T/Tp$ for some $p \in T$ with $p$
irreducible, i.e., $p$ has no factorization into a product of terms of positive
degree.  For nonzero $f\in T$ with $\deg(f) >0$, the division algorithm 
shows that every nonzero element of $T/Tf$ has the form $s+ Tf$ for some $s\in T$
with $\deg(s) < \deg(f)$.  Using this, it is easy to check 
(cf.~\cite[Lemma~5.2]{hazw1})
that for nonzero 
$f, g\in T$ of  positive degree
\begin{align}\label{eq:iso} 
\begin{split}
T/Tf \, \cong T/Tg\quad \text{iff} \quad&\deg(f) = \deg(g) \text{ and 
there exist nonzero } s,t\in T 
\text{ with }\\
&\deg(s) = \deg(t) <\deg(f) \text{ and } ft = sg.
\end{split} 
\end{align}

There is a divisor function 
$$
\delta\colon T\setminus \{0\} \to \DT \quad\ \  \text{given by }\  
\delta(f) = jh(T/Tf).
$$
It is easy to check that $\delta(fg) = \dd(f) + \dd(g)$, hence $\dd$
extends to a well-defined map $\dd\colon A^* \to \DT$ given by 
$\dd(fz\inv) = \dd(f) -\dd(z)$ for all $f\in T\setminus\{0\}$,
$z\in Z(T)\setminus\{0\}$. Clearly, $\delta(T)$ is the monoid of 
effective divisors in $\DT$.  Note also that $\deg(\dd(f)) = \deg(f)$
for all $f \in T \setminus \{0\}$.   
 It is proved in \cite[Prop.~5.3]{hazw1} that there 
is an exact sequence:
\begin{equation}\label{seq:Div}
1\, \longrightarrow\, [A^*,A^*]D^*\, \longrightarrow\, A^*\, \xrightarrow{\ \, \dd\, \ } \, 
\DT\, \longrightarrow\,0 
\end{equation}

Let $R = Z(T)$; so $R$ is the polynomial ring $K^{\sigma}[y]$, where $K^\sigma
= \{b\in K\mid \sigma(b) = b\}$ and $y = c\inv x^m$, with $m$  minimal such that
$\sigma^m|_K = \id$ and $c\in D^*$ satisfying $\sigma^m = \int(c)$ on $D$.
Note that $q(R) = Z(A)$.  This $R$ has its own divisor group $\Div(R)$, definable 
just as with $\DT$, and there is a corresponding divisor map $\dd_R\colon q(R)^* \to \Div(R)$.
The reduced norm map $\Nrd_A\colon A \to Z(A)$ maps $T$ to $R$, (as $T$~is integral
over $R$, and $R$ is integrally closed), and it is shown in \cite[Prop.~5.4]{hazw1} that 
there is a corresponding \lq\lq reduced norm" map $\DNrd\colon \DT \to \Div(R)$
which is injective, and such that the following diagram commutes: 
\begin{equation}\label{DNrddiag} 
\begin{CD}
A^* @>\dd>> \DT\\
@V\Nrd_AVV @V\DNrd VV \\
q(R)^* @>\dd_R>> \Div(R)
\end{CD}
\end{equation}

Of course, $T$ is a graded ring, with $T_i = Dx^i$ for all 
nonnegative $i\in \Z$.  Suppose $\rho\colon T \to T$ is a graded automorphism of 
$T$, i.e., a ring automorphism such that $\rho(Dx^i) = Dx^i$ for all $i$.
Then, $\rho$ restricts to a graded automorphism of $R$ and also determines 
ring automorphisms of $A$ and $q(R)$, all denoted $\rho$.  

\begin{lemma}\label{auto}
The graded automorphism $\rho$ of $T$ determines an automorphism of $\DT$ mapping 
$\SSS$ to ~itself, and also determines an automorphism of $\Div(R)$.  The automorphisms
determined by $\rho$ on each of the terms in diagram \eqref{DNrddiag} are 
compatible with the maps in that diagram.  
\end{lemma}

\begin{proof}
Since $\rho(D^*) = D^*$ and $\rho([A^*,A^*]) = [A^*,A^*]$, the exact sequence
\eqref{seq:Div} shows that $\rho$ determines a well-defined automorphism  $\rho$ of 
$\DT$ given by $\rho\dd(a)= \dd(\rho(a))$ for all $a\in A^*$.  Any class  
in $\SSS$ has the form $[T/Tp]$ for some irreducible $p$ in $T$.  Then,
$\rho(p)$ is also irreducible in $F$; so, ${\rho[T/Tp] = [T/T\rho(p)] \in \SSS}$.  
For~$\Div(R)$, $\rho$ is defined analogously by $\rho(\dd_R(q)) = \dd_R(\rho(q))$
for all $q\in q(R)^*$.  Since $\rho$ is an automorphism of $A$, we have 
$\rho(\Nrd_A(a)) = \Nrd_A(\rho(a))$ for all $a\in A$.  For, $\Nrd_A(a)$
is determined by the constant term of the minimal polynomial
$p_a\in Z(A)[X]$ of $a$ over $Z(A)$, and $p_{\rho(a)} = 
\rho(p_a)$.  Thus, we have $\rho\circ \dd = \dd\circ \rho$, 
$\rho\circ\dd_R = \dd_R\circ \rho$, and $\rho\circ \Nrd_A = \Nrd_A \circ \rho$.  
Since diagram \eqref{DNrddiag} commutes and $\dd$ is surjective, it follows that 
$\rho\circ 
\DNrd = \DNrd\circ \rho$.
\end{proof}

Suppose $\tau\colon T \to T$ is a unitary graded involution.  That is, $\tau$
is a ring antiautomorphism of $T$ with $\tau^2 = \id_T$, $\tau|_R \ne \id$, and 
$\tau(Dx^i) = Dx^i$ for all $i$.  (The last condition is equivalent to:
$\tau(D) = D$ and $\tau(x) = dx$ for some $d\in D^*$.)  This~$\tau$ extends
to a unitary involution on $A$ given by $\tau(fr\inv) = \tau(r)\inv \tau(f)$
for all $f\in T$, $r\in R\setminus \{0\}$.  Since $\tau([A^*,A^*]) = [A^*,A^*]$
and $\tau(D^*) = D^*$, the exact sequence \eqref{seq:Div} above shows
that $\tau$ determines an automorphism of order at most $2$ of $\DT$, also denoted
$\tau$.  This map $\tau\colon \DT \to \DT$ is given by $\tau(\dd(a))
= \dd(\tau(a))$ for all $a\in A^*$.  (For a simple left $T$-module~$S$,
we have $S\cong T/Tp$ for some irreducible $p\in T$.  Then, $\tau(p)$ is irreducible
in $T$, and $\tau[S] = [T/T\tau(p)]$.  It may seem surprising that this is well-defined,
independent of the choice of $p$, but the well-definition follows easily from 
\eqref{eq:iso}. In terms of two-sided ideals, $\tau[S]$ is the isomorphism class of 
simple left $T$-modules with annihilator $\tau(\ann_T(S))$.) Let
 $\DT^\tau = \{\alpha\in \DT
\mid \tau(\alpha) = \alpha\}$.  

\begin{lemma}\label{lem:invDiv} 
Let $\tau$ be a unitary graded involution on $T$.  
Suppose $\tau|_D \ne \id$ or $T$ is noncommutative.
Then, 
$$
\dd\big(\Sigma_\tau(A)\big)\, = \, \DT^\tau.
$$
\end{lemma}

\begin{proof}
Suppose first that $\tau|_D \ne \id$.  

Let $\Omega = \dd\big(\Sigma_\tau(A)\big) \subseteq \DT$.  Since $\delta$
maps generators of $\Sigma_\tau(A)$ into $\DT^\tau$, we have
${\Omega \subseteq \DT^\tau}$.  We must prove that this inclusion is equality.
Suppose that $\OM \subsetneqq \DTt$.

Note  that if $\alpha \in \DT$, 
say $\alpha = \dd(a)$, then $\alpha + \tau(\alpha) = \dd(a\tau(a)) \in \Omega$.
Take any $\eta\in \DTt \setminus \OM$.  We can write $\eta = \alpha - \beta$, 
where $\alpha$ and $\beta$ are effective divisors.   Then, 
$$
\alpha + \tau(\beta) \, = \, \eta + (\beta +\tau(\beta)) \, \equiv \, \eta\ 
(\modd\, \Omega). 
$$
So, $\alpha +\tau(\beta)$ is an effective divisor in $\DTt \setminus \OM$.

Let $\xi$ be an effective divisor in $\DTt\setminus \OM$ of minimal degree.  
Necessarily $\deg(\xi) >0$, as $\xi \ne 0$.  Say $\xi = \dd(z)$ for some 
$z\in T$.  So, $\deg(z) = \deg(\xi) >0$.  Factor $z$ into irreducibles in $T$,
say $z = p_1\ldots p_\ell$, and let $\pi_i = \dd(p_i)$.   
So,
$$
\pi_1 +\ldots +\pi_\ell \, = \, \eta \,=\, \tau(\eta)\, = 
\, \tau(\pi_1) + \ldots + \tau(\pi_\ell).
$$
Because the $\pi_i$ and $\tau(\pi_i)$ are all part of the $\Z$-base of
the free abelian group $\DT$, we must have ${\pi_1= \tau(\pi_j)}$ for some 
index $j$.  Suppose $j>1$.  Then, $\pi_1 +\pi_j = \pi_1 + \tau(\pi_j) \in \OM$.
Let ${\xi' = \xi -(\pi_1 + \pi_j)}$. Then, $\xi' \equiv \xi\ (\modd \, \OM)$. But, 
$\xi'$ is an effective divisor with $\deg(\xi') <\deg(\xi)$.  This contradicts
the minimality of $\deg(\xi)$.  So, we must have $j = 1$, i.e., $\tau(\pi_1) = \pi_1$.  
If $\pi_1\in \OM$, then $\xi -\pi_1$ is an effective divisor in $\DTt \setminus \OM$
of degree less than $\deg(\xi)$.  Hence $\pi_1 \in \DTt\setminus \OM$, 
so in fact $\xi = \pi_1$ by the minimality of $\deg(\xi)$.  

To simplify notation, let $p = p_1$ and $\pi = \pi_1$.  Since $p$ and $\tau(p)$
are irreducible, we have 
$\pi = [T/Tp]$ and $\tau(\pi) = [T/T \tau(p)]$.  Hence, the equality 
$\pi = \tau(\pi)$ implies that $T/ Tp \cong T/T\tau(p)$.  Therefore, by~ 
\eqref{eq:iso}, there exist $f, g \in T\setminus\{0\} $ with $\deg(f) = 
\deg(g) <\deg(p)$ and
\begin{equation}\label{eq:pf}
pf\, = \, g\tau(p).
\end{equation}

Suppose first that $\tau(f) = g$.  Then, $pf = \tau(f) \tau(p) = \tau(pf)$, 
so $pf \in \Sigma_\tau(A)$, hence $\dd(pf)\in \OM$.  
Then, ${\dd(f) = \dd(pf) -\pi \equiv -\pi \,(\modd \, \OM)}$.  
Therefore, $\dd(f)$ is an effective divisor in $\DTt\setminus \OM$ with 
$$
\deg(\dd(f))  \, = \,  \deg(f)  \, < \,  \deg(p)  \, = \,  \deg(\pi)  \, =  
\, \deg(\xi).
$$
This contradicts the minimality of $\deg(\xi)$.  So, we must have $\tau(f) \ne g$.

Now suppose that $\tau(f) \ne -g$.  Then, let $f' = f+\tau(g) \in T\setminus\{0\}$,
and $g' = \tau(f') = \tau(f) + g$. By applying $\tau$ to \eqref{eq:pf}, we have
$p\tau(g) = \tau(f) \tau(p)$, which when added to \eqref{eq:pf} yields 
$pf' = g' \tau(p)$.   
But then the preceding argument for the case $\tau(f) = g$ applies for $f'$, 
as $f' \ne 0$ and $\tau(f') = g'$; it shows that $\dd(f') \in \DTt \setminus \OM$.
Since $\deg(f') \le \max\big(\deg(f), \deg(\tau(g)\big) = \deg(f) <\deg(\xi)$, this 
contradicts the minimality of $\deg(\xi)$.  

We thus have $\tau(f) = - g$ while $\tau(f) \ne g$.  Hence, $\charac(D) \ne 2$.  
 Since  we have assumed $\tau|_D \ne \id$, there is $c \in D^*$ with $\tau(c)
= -c$.  Then,
\begin{equation}\label{eq:pfc}
pfc\, = \, g\tau(p) c \, = \, c\inv cg\tau(p) c \, = \, c\inv \tau(pfc)  c.
\end{equation}
Since $\tau(c)  = -c$, 
 the map $\tau' = \int(c\inv)\circ \tau$ is 
a unitary involution on $A$ with $\tau'\sim \tau$.  Hence, by 
Lemma~\ref{lem:equivrem}(a),   $\Sigma_{\tau'}(A) = \Sigma_{\tau}(A)$.   
Since equation~\eqref{eq:pfc} shows $pfc = \tau'(pfc)$, we thus have\break 
${pfc\in \Sigma_{\tau'}(A) =\Sigma_\tau(A)}$, and hence $\dd(pfc) \in \OM$.  
This shows
that $\dd(fc) \equiv -\pi\ (\modd \, \OM)$, and hence 
$\dd(fc) \in\DTt \setminus \OM$.  But, $\dd(fc)$ is an effective divisor, 
with ${\deg(\dd(fc)) = \deg(f) < \deg(\pi) = \deg(\xi)}$.  Thus, we have a
contradiction to the minimality of $\deg(\xi)$.  The contradiction arose 
from the assumption that $\OM \subsetneqq \DTt$.  So, we must
have $\DTt = \OM =\dd\big(\Sigma_\tau(A)\big)$, as desired.

We have thus far assumed that $\tau|_D \ne \id$.  Now, suppose that $\tau_D = \id$.
Then, $D$ must be commutative, since $\tau$ is an antiautomorphism.  
Our hypothesis now is that $T$ is noncommutative; therefore,
 ${\sigma = \int(x)|_D \ne \id}$.  
Since, $\tau(Dx) = Dx$ there is a nonzero $d\in D$ with $\tau(dx) = \pm dx$.
Then let $\wi \tau = \int(dx) \circ \tau$, which is a unitary graded involution 
of $T$ with $\wi \tau \sim \tau$ as unitary involutions on $A$.  Hence, 
$\Sigma_{\wi\tau}(A) = \St A$ by Lemma~\ref{lem:equivrem}(a).  
Also, for any $a\in A^*$, 
$$
\wi\tau (\dd(a)) \, = \, \dd(\wi \tau(a)) \, = \, \dd(dx \, \tau(a)\hsp (dx)\inv) 
\, = \, \dd(dx) + \dd(\tau(a)) -\dd(dx) \, = \,\dd(\tau(a)) \, = \, \tau(\dd(a)). 
$$
Thus, $\wi \tau$ and $ \tau$ have the same action on $\DT$.  
Furthermore, $\wi \tau|_D = \int(dx)|_D = \sigma \ne \id$. Therefore, the preceding 
argument shows that the Lemma holds  for $\wi \tau$.  This yields for $\tau$,
$$
\dd\big(\St A\big) \, = \, \dd\big(\Sigma_{\wi\tau}(A)\big) \, = \, \DT^{\wi\tau}
\, = \, \DT^\tau.
$$  
Thus, the lemma holds in all cases.  
\end{proof}

\begin{remark}
The assumption that $\tau|_D \ne \id$ or  $T$ is noncommutative  is  
definitely needed for Lemma~\ref{lem:invDiv}.  For example, suppose that 
$T$ is commutative 
and $\tau$ is defined by $\tau|_D = \id$ and $\tau(x) = -x$.  Then, 
${\dd(x) \in \DT^\tau\setminus \dd(\St A)}$. 

\end{remark}

Now we combine the action of the graded unitary involution 
$\tau$ on $\DT$ with a finite group action.  Let $H$ be a finite
abelian group which acts on the set $\SSS$ of isomorphism classes
of simple left $T$-modules.  This action induces an action of 
$H$ on $\DT$, making $\DT$ into a permutation module for $H$.  There
is an associated norm map $N_H\colon \DT \to \DT$ given by 
$N_H(\alpha ) = \sum_{h\in H} h \alpha$.  Suppose the actions of
$H$ and $\tau$ on $\DT$ are related by
\begin{equation}\label{eq:tauvsh}
\tau(h\alpha) \, = \, h\inv \tau(\alpha)\qquad
\text{for all }\, h\in H, \, \alpha \in \DT. 
\end{equation}

\smallskip
\begin{lemma}\label{lem:NH}
\qquad$
\{\alpha\in \DT\mid N_H(\alpha)= N_H(\tau(\alpha))\} \, = \, 
I_H(\DT) + \sum\limits_{h\in H}\DT^{h\tau},
$\\[4pt]
\noindent where $I_H(\DT) = \big\langle h\alpha - \alpha\mid h \in H, \, 
\alpha \in \DT\big \rangle$
and $\DT^{h\tau} = \{ \alpha\in \DT \mid h\tau(\alpha) = \alpha\}$.   
\end{lemma}

\begin{proof}
Let $G$ be the semidirect product group $H \rtimes_\psi\langle \tau\rangle$
built using the homomorphism ${\psi \colon \langle \tau\rangle \to \Aut(H)}$ given by 
$\psi(\tau)(h) = h\inv$.  (This is well-defined as $H$ is abelian.) More explicitly,
$G = H\cup H\tau$ (disjoint union), where the multiplication  is determined by that 
in $H$ together with $\tau h = h\inv \tau$ and $\tau^2 = 1$.
Note that $G$ is a generalized dihedral group in the terminology of \cite[\S\,2.4]{hazw2}
since every element of $G\setminus H$ has order $2$.   
The hypothesis \eqref{eq:tauvsh} shows that the actions of $\tau$ and $H$ on 
$\DT$ combine to yield a $\Z$-linear group action of $G$ on $\DT$.  Also,
$G$ sends $\SSS$ to $\SSS$, since this is true for $H$ and $\tau$.  Thus,
$\DT$ is a permutation $G$-module.  From this group action 
we build a new action of~$G$ on~$\DT$, denoted~$*$, given by 
$$
h*\alpha \, = h\alpha\quad \text {and} \quad (h\tau) * \alpha = - h\tau \alpha\qquad
\text{for all } h\in H, \ \alpha \in \DT.
$$
Let $\widetilde{\DT}$ denote $\DT$ with this twisted $G$-action, 
and let $\widetilde {N_G} \colon \tDT \to\tDT$ be the associated norm map,
given by
$$
\tNG(\alpha) \, =\, \tsum\limits_{h\in H} h\alpha - \sum \limits_{h\in H}
h\tau\alpha\, = \, N_H(\alpha) - N_H(\tau\alpha)\qquad\text{for all } \alpha\in \DT.  
$$ 
Thus, the Lemma describes $\ker(\tNG)$.  

Note that while $\DT$ is a permutation $G$-module, the twisted $G$-module 
$\tDT$ need not be a permutation $G$-module, as the twisted action of 
$G$ does not map $\SSS$ to itself.  Write $\SSS = \bigcup_{j\in J}\OO_j$, 
where the $\OO_j$ are the distinct  $G$-orbits of $\SSS$ (for the original $G$-action).  
Then, $\DT = \bigoplus_{j\in J} M_j$, where each $M_j = \bigoplus_{s \in \OO_j}\Z s$, 
which is a $G$-submodule of $\DT$.  Let $\OO$ be one of the orbits, and let 
$M = \bigoplus_{s \in \OO}\Z s$.  Take any $s\in \OO$, and let
$\VV = H \cdot s$, which is an $H$-orbit within $\OO$.   
Then, $\tau\VV = \tau H\cdot s = H\tau \cdot s$, which is also an $H$-orbit in $\OO$,
with $\VV \cup \tau\VV = Hs\cup \tau H s = Gs = \OO$.  Let 
$\VV = \{s_1, \ldots, s_n\}$.  There are two possible cases:

Case I.  $\VV \cap \tau \VV = \varnothing$.  
Then, $\{s_1, \ldots, s_n , -\tau(s_1),
\ldots, -\tau(s_n)\}$ is a $\Z$-base of $M$ which is mapped to itself by the 
$*$-action of $G$.  Thus, $\widetilde M$ (i.e., $M$ with the twisted $G$-action)
is a permutation $G$-module.  Consequently, 
$\widehat H\inv(G, \widetilde M) = 0$ (as is true for all $G$-modules).  This means
that 
$$
\ker(\tNG)\cap M \, = \, I_G(\widetilde M) \, = 
\, \langle g*m - m\mid g\in G,\ m\in M\rangle
\, = \, \langle hm -m,\, - h\tau m - m\mid h\in H,\ m\in M\rangle.
$$
Each $hm - m \in I_H(M) \subseteq I_H(\DT)$, while 
$-h\tau m - m \in M^{h\tau}\subseteq \DT^{h\tau}$.  
Thus, 
$$
\ker(\tNG) \cap M  \ \subseteq \ I_H(\DT) + \tsum\limits_{h\in H}\DT^{h\tau}.
$$

Case II.  $\VV \cap \tau \VV \ne \varnothing$.  Then, $\VV = \tau \VV$, since they
are each $H$-orbits.  So, for each $t \in \VV$ there is $k\in H$ with 
$t = \tau k\inv  t = k\tau\, t $; so, $t\in M^{k\tau}$.  Thus,
$$
\ker(\tNG) \cap M \,\subseteq M \,= \,\tsum \limits_{h\in H}M^{h\tau}
\,\subseteq \,\tsum \limits_{h\in H}\DT^{h\tau}.
$$

By combining the two cases, we obtain
$$
\ker(\tNG) \, = \,\tsum_{j\in J}\big(\ker(\tNG)\cap M_j\big) 
\, \subseteq \,I_H(\DT) + \sum_{h\in H} \DT^{h\tau}.
$$
This proves $\subseteq$ in the Lemma, and the reverse inclusion is clear.
 \end{proof}

\bigskip

\section{Proof of the theorem}

We can now prove Theorem~\ref{th:main}.

\begin{proof}
Let $\grE$ be a graded division ring (finite-dimensional over its center) with 
a unitary graded involution $\tau$, and let $Q = q(\grE)$.
Because $\Nrd_Q|_\grE = \Nrd_\grE$, as noted in \S\,2, we have $\Stp \grE
\subseteq \Stp Q$. Also, clearly $\St \grE \subseteq \St Q$.  Therefore,
there is a canonical map $\varphi_\grE\colon \SK(\grE, \tau) \to \SK(Q, \tau)$.
We will show that $\varphi_\grE$ is an isomorphism. 

To see that $\varphi_\grE$ is injective, give $\Ga_\grE$ a total ordering making it 
into a totally ordered abelian group.  Then define a function $\lambda\colon
\grE \setminus\{0\}\to \grE^*$ as follows:  For $c\in \grE\setminus \{0\}$,
write $c = \sum _{\ga \in \Ga_\grE}c_\ga$ where each
$c_\ga \in \grE_\ga$.  Then set $\lambda(c) = c_\varepsilon$, where 
$\varepsilon$ is minimal such that $c_\varepsilon\ne 0$.  It is easy
to check that $\lambda(cd) = \lambda(c)\lambda(d)$ for all $c,d\in 
\grE \setminus\{0\}$. It follows that $\lambda$ extends to a well-defined
group epimorphism $\lambda \colon Q^* \to \grE^*$ given by 
$\lambda(cz\inv) = \lambda(c) \lambda(z)\inv$ for all $c\in \grE \setminus\{0\}$,
$z\in Z(\grE) \setminus\{0\}$.  Clearly $\lambda|_{\grE^*} = \id$.  Note that as 
$\tau$ is a graded involution, we have $\lambda(\tau(c)) = \tau(\lambda(c))$
for all $c\in \grE \setminus\{0\}$, and hence for all $c\in Q^*$.  Now take any 
$u \in \St Q \cap \grE^*$.  Then, $u = s_1\ldots s_m$ for some $s_i\in Q^*$ with 
each $\tau(s_i) = s_i$. Then, $u = \lambda(u) = \lambda(s_1) \ldots \lambda(s_m)$
with $\tau(\lambda(s_i)) = \lambda(\tau(s_i)) = \lambda(s_i)$; so, $u \in \St \grE$.
This shows that $\St Q \cap \grE^* \subseteq \St\grE$.  Therefore, the canonical map 
$\varphi_\grE$ is injective.
 
To show that $\varphi_\grE$ is surjective, we will prove the following:
\begin{equation}\label{claim}
\Stp Q \, = \, \big( \Stp Q \cap\grE_0^*\big) \cdot \St Q
\end{equation}
Note that surjectivity of $\varphi_\grE$ follows immediately from \eqref{claim}
because $\Stp Q \cap \grE_0^* \subseteq \Stp \grE$.   

We next reduce to the case of a finitely-generated grade group.  For any graded 
division algebra~$\grE$, observe that if  
$\Delta$ is a subgroup of~$\Ga_\grE$, then $\grE$ has the  graded division subalgebra
$\grE_\Delta = \bigoplus_{ \varepsilon \in \Delta}\grE_\varepsilon$, which is 
finite-dimensional over its center with $\Ga_{\grE_\Delta} = \Delta$. 
 Let $Q_\Delta = q(\grE_\Delta) \subseteq Q$. Take any $\ga \in \Gamma_{Z(\grE)}$ 
such that $\tau|_{Z(\grE)_\ga} \ne \id$ and any $c_1, \ldots, c_m \in \grE^*$
such that the $c_i$ span $\grE$ as a graded $Z(\grE)$-vector space. Then, 
the $c_i$ also span $Q$ as a $Z(Q)$-vector space. So, if $\Delta$ is any 
finitely-generated subgroup of $\Gamma_\grE$ such that $\ga \in \Delta$ and
$\deg(c_1), \ldots, \deg(c_m)\in \Delta$, then $\tau|_{\grE_\Delta}$ is a unitary 
graded involution and  the map $Q_\Delta \otimes_{Z(Q_\Delta)} Z(Q) \to Q$
is surjective.  This map is also injective, as its domain is 
simple since $Q_\Delta$ is central simple and finite-dimensional over 
$Z(Q_\Delta)$.  Therefore, 
$\DIM {Q_\Delta}{Z(Q_\Delta)} = \DIM Q{Z(Q)}$, and hence $\Nrd_{Q_\Delta} = 
\Nrd_Q|_{Q_\Delta}$, so $\Stp{Q_\Delta} = \Stp Q \cap Q_\Delta^*$.
Now, $\grE = \varinjlim \grE_\Delta$ and  $\Stp Q = \varinjlim \Stp{Q_\Delta}$ 
as $\Delta$ ranges over   finitely-generated subgroups of $\Ga_\grE$
containing $\gamma$ and $\deg(c_1), \ldots, \deg(c_m)$;  therefore, 
equality \eqref{claim} holds for $\grE$ if it holds for each $\grE_\Delta$.  
Thus, it suffices to prove \eqref{claim} for $\grE$ with $\Ga_\grE$ 
finitely-generated. 

Henceforward, suppose $\Ga_\grE$ is a finitely generated, hence free,
abelian group, say $\Ga_\grE = \Z \ga_1 \oplus \ldots \oplus \Z\ga_n$.
Take any nonzero $x_i \in \grE_{\ga_i}$ for $i = 1, 2, \ldots, n$.  
Then, $\grE = \grE_0[x_1, x_1\inv, \ldots. x_n, x_n\inv]$, an iterated
twisted Laurent polynomial ring over the division ring $\grE_0$.
We prove \eqref{claim} by induction on $n$. 

Suppose first that $n = 1$.  Let $D= \grE_0$ and $x= x_1$, and let $T = 
D[x]$, which is a graded  subring of $\grE$.  So, $T$ is the twisted
polynomial ring $D[x;\sigma]$, where $\sigma =\int(x)|_D$.  Then, 
$\sigma|_{Z(D)}$ has finite order since $\DIM{Z(D)}{Z(\grE)_0}<\infty$.  
To be consistent with the notation of \S\,3, let $A = q(T) = q(\grE) = Q$.   
Since $\tau$ is a graded unitary involution on $\grE$, its restriction 
$\tau|_T$ is a graded
unitary involution on $T$.  Let $\DT$ and $\delta \colon A^* \to \DT$ be 
as in \S\,3.  If $T$ is commutative, then $A$ is also commutative, 
so $\SK(A) = 1$ and \eqref{claim} holds trivially.  Thus, we may assume that
 $T$ is noncommuative.  Take any  
$a\in \Stp A$.  So, $\Nrd_A(a) = \tau(\Nrd_A(a)) = \Nrd_A(\tau(a))$ by
\eqref{Nrdcompat}. 
Recall from \cite[Remark~5.1(iv)]{hazw1} that $\dd(\Nrd_A(a)) = r\dd(a)$, 
where $r = \ind(A)$.  Hence,
$$
r\,\tau(\dd(a)) \, = \, r \,\dd(\tau(a)) \, = \, \dd(\Nrd_A(\tau(a))
\, = \, \dd(\Nrd_A(a)) \, = \, r\,\dd(a).
$$
Then, $\tau(\dd(a)) = \dd(a)$, as $\DT$ is torsion-free.  By 
Lemma~\ref{lem:invDiv}, we have $\dd(a) = \dd(b)$ for some 
$b\in \St A$.  It then follows from the exact sequence
\eqref{seq:Div} that $a = bcd$ for some $c\in [A^*, A^*]$ and 
$d\in D^*$.  Recall that $[A^*,A^*] \subseteq \St A$ by  
Lemma~\ref{lem:equivrem}(b). So,    
$bc\in \St A \subseteq \Stp A$.  Hence, $d = (bc)\inv a \in \Stp A
\cap D^*$, so $a = (bc) d \in \St A (\Stp A \cap D^*)$.  
This proves the inclusion $\subseteq$ in \eqref{claim}. 
The reverse inclusion is clear.
  Thus, 
\eqref{claim} holds  when $n = 1$. 

Now assume $n>1$.  By induction \eqref{claim} holds for all graded division algebras 
with grade group of rank less than $n$. Let $x = x_1$ as above, and let 
\begin{align*}
\grZ \,&= \, Z(\grE); \qquad\qquad\qquad\qquad\qquad\qquad\qquad\qquad\qquad
\qquad\qquad\qquad\qquad\qquad\qquad\\
D\, &= \,\grE_0;\\
T\,&= \, D[x];\\
R\,&= \, Z(T);\\
A\,&= \, q(T);\\
\grB\,&= \, A[x_2,x_2\inv, \ldots, x_n, x_n\inv] \, \subseteq \, Q;\\
\grC\, &= \, Z(\grB).
\end{align*}
For a clearer picture of $\grB$ and $\grC$, let $\grU = \grE_0[x_1, x_1\inv]$, which 
is a graded division subalgebra of $\grE$ with $\grU_0 = \grE_0$, 
$\Ga_\grU = \Z\gamma_1$, and $q(\grU) = A$. Let $\grF = \grU \cap \grZ$, which is a graded field 
with $\grF_0 = \grZ_0$ and $\Ga_\grF = \Ga_\grZ \cap \Ga_ \grU$.  So, 
$$
\IND{\Ga_\grU}{\Ga_\grF}\, = \, \IND{\Ga_\grE \cap\Z\gamma_1\,}
{\, \Ga_\grZ \cap \Z\ga_1} \, \le \, \IND{\Ga_\grE}{\Ga_\grZ} \, < \infty,
$$
and hence $\DIM \grU\grF = \DIM {\grU_0}{\grF_0}\, \IND{\Ga_\grU}{\Ga_\grF}
= \DIM {\grE_0}{\grZ_0}\, \IND{\Ga_\grU}{\Ga_\grF} <\infty$.  Since $\grF$
is central in $\grU$, we have\break
 ${A = q(\grU) = q(\grF) \otimes _\grF \grU}$.
Hence, $\grB = q(\grF) \otimes_\grF \grE$, so $\grC = Z(\grB) = 
q(\grF) \otimes_\grF \grZ$, from which it is clear that the rank of $\grB$
as a free $\grC$-module equals $\DIM\grE \grZ$, which is finite.  Let 
$\grE' = \grU[x_2, x_2\inv, \ldots, x_n , x_n\inv] = \grE$, but we
grade~$\grE'$ so that $\grE'_0 = \grU$, $\Ga_{\grE'} = \Z\gamma_2\oplus \ldots
\oplus \Z\ga_n$, and 
$\grE'_ \varepsilon = \bigoplus_{j\in \Z}E_{j\ga_1+\varepsilon}$,
for each $\varepsilon \in \Ga_{\grE'}$. Now, $\grB$~is obtained from
$\grE'$ by inverting nonzero elements of $\grF$, which are all central
in $\grE'$ and homogeneous of degree $0$.  So, the grading on $\grE'$
extends to a grading on $\grB$ with $\grB_0  = A$ and 
$\Ga_\grB = \Ga_{\grE'} = \Z\ga_2 \oplus \ldots \oplus \Z \ga_n$; 
this grading restricts to a grading on $\grC = Z(\grB)$.
Note that $\grB$ is a graded division ring since $\grB_0 = A$ is a 
division ring and each homogeneous component $\grB_\varepsilon$ of 
$\grB$ contains a unit
of $\grB$ (namely any nonzero homogeneous element of $\grE$ in 
$\grB_\varepsilon$).    We saw above that $\DIM\grB \grC< \infty$.
Also, $q(\grB) = Q$, as $\grE \subseteq \grB \subseteq Q = q(\grE)$.
Our unitary involution~$\tau$ is a graded involution for $\grE$, 
so also a graded involution for $\grE'$, hence also a graded involution
for $\grB$.  Moreover, $\tau$ is a unitary involution for $\grB$,
as $\tau|_\grZ \ne \id$ and $\grZ = Z(\grE) \subseteq \grB \cap Z(Q) 
\subseteq  Z(\grB)$.  Since $\Ga_\grB$ has $\Z$-rank $n-1$, 
\eqref{claim} holds for $\tau$ on $\grB$ by induction,  
i.e.
$$
\Stp Q \, = \, \big( \Stp Q \cap A^*\big)\cdot \St Q. 
$$  
Therefore, to prove
\eqref{claim} for $\grE$, it suffices to prove 
\begin{equation}\label{newclaim}
\Stp Q \cap A^* \, \subseteq \big( \Stp q \cap D^*\big ) \cdot \St Q.
\end{equation}
For this, we will work with $\DT$ and the action on it by $\tau$ and 
various automorphisms.  Note that while $\tau$ is unitary as an involution 
on $\grB$, its restriction to $A$ need not be unitary.  But, if $\tau|_A$ 
is not unitary, then $\tau|_{Z(A)} = \id$, so $\tau|_{\grC_0} = \id$, hence
 $\grC$ is totally ramified over $\grC^\tau$, i.e., $\grC_0 = (\grC^\tau)_0$.  
But then $\SK(\grB, \tau) = 1$
by \cite[Th.~4.5]{hazw2}.  So, $\SK(Q,\tau)= \SK(B,\tau) = 1$ since \eqref{claim}
holds for $\grB$, whence \eqref{newclaim} and \eqref{claim} for $\grE$ obviously hold.  
Thus, we will assume from now on that $\tau|_A$ is unitary.

Assume  that $T$ is not commutative.    

Since $\grB$ is a  graded division algebra and $\grC = Z(\grB)$, it is known
(see \S\,2)
that $Z(\grB_0)$ is Galois over $\grC_0$, and there is a well-defined
epimorphism $\theta_\grB \colon \Gal(Z(\grB_0)/ \grC_0)$ 
given by $\varepsilon \mapsto \int(b_\varepsilon)|_{Z(\grB_0)}$ for any 
nonzero $b_\varepsilon$ in $\grB_\varepsilon$.  Here, $\grB_0 = A$.
Let $H = \Gal(Z(A) /\grC_0)$, which is abelian, as $\theta_\grB$ is surjective.  
We next define a group action of $H$ on $\DT$ so as to be able to use
Lemma~\ref{lem:NH}.

We use the notation associated to $T$ and $\DT$ as in \S\,3, including the 
epimormphism 
${\dd\colon A^* \to \DT}$.
Observe that conjugation by an element of $\grE^*$ yields a degree preserving 
automorphism of  $T$, and hence a well-defined automorphism of $\DT$.  That is, 
for any $y\in \grE^*$ and $a\in A^*$, define 
\begin{equation}\label{Eaction}
y\cdot \dd(a) \, = \dd(yay\inv).  
\end{equation}
This is well-defined because $\ker(\dd) = [A^*, A^*] D^*$ (see \eqref{seq:Div}),
which $\int(y)$ maps to itself.  For any irreducible $p$ of the twisted polynomial
ring $T$, its conjugate $y p y\inv$ is also irreducible in $T$, and ${y\cdot [T/Tp] = y\cdot \dd (p)
= [T/Typy\inv]}$.  Hence,  the action of $y$ maps the distinguished $\Z$-base 
$\SSS$ of $\DT$ to itself.  Therefore, the group action of $\grE^*$ on 
$\DT$ given by \eqref{Eaction} makes $\DT$ into a permutation $\grE^*$-module.
Note that if $u\in \grE^*$ and $\int(u) |_{Z(A)} = \id$, then by Skolem-Noether there is 
$c\in A^*$ with 
$\int(u)|_A = \int(c)$.  Then, for $a\in A$, 
\begin{equation}\label{conjbyu}
u\cdot \dd(a) \, = \, \dd(cac\inv) \, = \, \dd(c) +\dd(a) - \dd(c) \, = \, \dd(a),
\end{equation}
so $u$ acts trivially on $\DT$.  Now, for any $y\in \grE^*$, we have 
$\tau\circ \int(y) \circ\tau\circ \int(y) = \int(\tau(y) \inv y)$.
Since $\tau$ preserves degrees, we have $\tau(y) \inv y \in \grE_0 = D \subseteq A$,
hence $\int(\tau(y) \inv y)|_{Z(A)} = \id$.
Therefore, \eqref{conjbyu} with   $u = \tau(y)\inv y$ shows that 
$\tau y\tau y \cdot \alpha = \alpha$ for all $\alpha \in \DT$, i.e., 
\begin{equation}\label{ytauaction}
\tau y \cdot \alpha\, = \, y\inv \tau \cdot \alpha. 
\end{equation} 
We obtain from the $\grE^*$-action an induced action of $H$ on $\DT$:
For any $h \in H$, choose any $y_h \in \grE^*$ such that  $\int(y_h)|_{Z(A)} = h$.  
(Such a $y_h$ exists because the composition $\grE^* \to \grB^* \to \Ga_\grB$ is 
surjective, as is $\theta_\grB\colon \Ga_\grB \to H$.) Then, for $\alpha \in \DT$, 
set $h \cdot \alpha = y_h \cdot \alpha$.  If we had made a different choice of $y_h$,
say~$y_h'$, then as $\int(y_h^{\prime -1} y_h)|_{Z(A)} = h\inv h= \id$, the calculation 
in \eqref{conjbyu} shows that $y_h'\cdot \alpha = y_h\cdot \alpha$ for all 
$\alpha \in \DT$.  Thus, the action of $H$ on $\DT$ is well-defined, independent of 
the choice of the $y_h$.  So, $\DT$ is a permutation $H$-module, and from 
\eqref{ytauaction}, we have $\tau h\cdot \alpha = h\inv \tau\cdot \alpha$ for all
$h\in H$ and $\alpha \in \DT$.  Therefore, Lemma~\ref{lem:NH} applies for the 
norm map $N_H$ of $H$ on $\DT$.   We interpret the terms appearing in that lemma:
First,
\begin{equation}\label{IH}
I_H(\DT) \, = \, \dd\big([\grE^*, A^*]\big).
\end{equation}
For, to see $\subseteq$, take any generator $h\cdot \beta - \beta$ of $I_H(\DT)$,
and take any $b \in A^*$ with $\dd(b) = \beta$.  Then,
$$
h\cdot \beta -\beta \, = \, \dd(y_h b y_h\inv) - \dd(b) \, = \,\dd(y_hby_h\inv b\inv)
\,\in\, \dd\big([\grE^*,A^*]\big).  
$$
For the reverse inclusion, take any $z\in \grE^*$  and $a\in A^*$.
Let $h = \int(z)|_{Z(A)}$.  Since we could have chosen~$z$ for $y_h$, we have 
$$
\dd(zaz\inv a \inv) \, = \, \dd(zaz\inv) - \dd(a) \, = \, h\cdot \dd(a) - \dd(a)
\, \in \, I_H(\DT). 
$$
This proves \eqref{IH}.

For any $h\in H$, we have selected $y_h \in \grE^*$ so that 
$\int(y_h)|_{Z(A)} = h$.  For any $d\in D = \grE_0$, we have 
$\int(dy_h)|_{Z(A)}= \int(y_h)|_{Z(A)}$. Since $\grE_{\deg(y_h)} = 
\grE_0 y_h$, the graded involution $\tau$ maps $\grE_0y_h$ to itself.
Hence, there is $d_h\in \grE_0$ with $\tau(d_hy_h) =\pm d_h y_h$. 
Then, replace $y_h$ by $d_h y_h$, and  set $\tau_h = \int(y_h) \circ \tau$.  
Since $\tau(y_h) = \pm y_h$, the map $\tau_h$ is a  unitary graded involution on
$\grE$ with $\tau_h \sim \tau$ on $\grE$.  So, $\tau_h$ restricts to a 
graded involution on $T$.  If $\tau_h|_T$ is not unitary, then $\tau_h|_A$ is not 
unitary, so $\SK(Q, \tau_h) = 1$ by the same argument as given above for $\tau$.
But, $\SK(Q, \tau_h) = \SK(Q, \tau)$ as $\tau_h\sim \tau$ by 
Lemma~\ref{lem:equivrem}(a), and the triviality  
of $\SK(Q,\tau)$ implies \eqref{newclaim} and \eqref{claim}.  Thus, we
are done if any $\tau_h|_A$ is not unitary.  We therefore assume that each $\tau_h|_A$ is
unitary, so  $\tau_h|_T$ is a graded unitary involution.  
 Since the action of~ 
$\tau_h$ on $\DT$ coincides with that of $h\tau$, Lemma~\ref{lem:invDiv}
yields
\begin{equation}\label{Divhtau}
\DT^{h\tau}\, = \, \dd\big(\Sigma_{\tau_h}(A)\big)\qquad \text{for each }\, h\in H.
\end{equation}
(The Lemma applies because we are assuming that $T$ is noncommutative.) 


We  now prove \eqref{newclaim}.  For this, take any $a\in \Stp Q \cap A^*$, 
and let $\alpha = \dd(a)$.
 Since $A = \grB_0$ and $Q = q(\grB)$, the norm formula \eqref{normformula} yields 
$$
\Nrd_Q(a) \, = \, N_{Z(A) /\grC_0}(\Nrd_A(a))^\lambda\qquad \text{where}\ \, 
\lambda = \ind(Q) \big/ \big(\DIM{Z(A)}{\grC_0} \cdot \ind(A)\big).
$$
Note that $N_{Z(A) /\grC_0}$ is just the norm map $N_H$ for the $H$-module $Z(A)$.
Here, $\DT$ and $\Div(R)$ are also $H$-modules, and the maps $\dd_R\colon Z(A) 
\to \Div(R)$ and $\DNrd\colon \DT \to \Div(R)$ in diagram~\eqref{DNrddiag}
are $H$-equivariant.  This 
follows from Lemma~\ref{auto} since for $h\in H$ the action of $h$ on $Z(A)$, $\Div(R)$,
and $\DT$ is induced via $\int(y_h)$, which is a graded automorphism of $T$.  
Therefore, $\dd_R$ and $\DNrd$ commute with the norm maps $N_H$ on 
$Z(A)$, $\Div(R)$, and $\DT$.
Thus, for $\alpha = \dd(a)$ we have, using commutative diagram \eqref{DNrddiag},
\begin{align*}
\dd_R(\Nrd_Q(a)) \, &= \, \dd_R(N_{Z(A)/\grC_0}(\Nrd_A(a))^\lambda
\, = \, \lambda \dd_R N_H(\Nrd_A(a))  \, = \,
\lambda N_H \dd_R(\Nrd_A(a))\\
&= \, \lambda N_H(\DNrd\dd(a)) \, = \, 
\lambda \DNrd( N_H(\alpha)).
\end{align*}
Since $a\in \Stp Q$, we have $\Nrd_Q(a) = \Nrd_Q(\tau(a))$. Thus,
\begin{equation}\label{DNrdNH}
\lambda\DNrd (N_H(\alpha)) \, = \, \dd_R(N_Q(a)) \, = \dd_R(N_Q(\tau(a))
\, = \, \lambda \DNrd(N_H(\tau(\alpha))).
\end{equation} 
But, $\Div(R)$ is a torsion-free $\Z$-module, and $\DNrd$ is injective by 
\cite[Prop.~5.4]{hazw1}.   Therefore, equation~\eqref{DNrdNH} yields $N_H(\alpha)
= N_H(\tau(\alpha))$.  By Lemma~\ref{lem:NH} and \eqref{IH} and \eqref{Divhtau},
we thus have 
$$
\alpha\, \in \, I_H(\DT) + \tsum \limits_{h\in H}\DT^{h\tau} \ = \ 
\dd\big([\grE^*, A^*]\big) + \tsum \limits_{h\in H}\dd\big( \Sigma_{\tau_h}(A)\big).
$$
Hence, there exist $b\in [\grE^*, A^*]$ and $c_h \in \Sigma_{\tau_h}(A)$ for each $h\in H$
with 
$$
\dd(a) \,=\,\, \dd(b) + \tsum\limits_{h\in H}\dd(c_h) \,=\,\, 
\dd(b\prod\limits_{h\in H}c_h).
$$
  Then, as $\ker(\dd) = [A^*,A^*]D^*$, there exist $d\in D^*$ and $e\in [A^*,A^*]$
such that $a = deb\prod_{h\in H}c_h$. Let $q= eb\prod_{h\in H}c_h$.
Since $[A^*,A^*] \subseteq \St Q$ by 
Lemma~\ref{lem:equivrem}(b), 
we have $eb\in \St Q$. But also, each ${c_h \in \Sigma_{\tau_h}(A)
\subseteq \Sigma_{\tau_h} (Q) = \St Q}$, with the last equality given by 
Lemma~\ref{lem:equivrem}(a), 
since $\tau_h \sim \tau$ as involutions on $Q$.  Hence, 
${q\in \St Q \subseteq\Stp Q}$.  So, ${d = aq\inv \in D^* \cap \Stp Q}$, whence 
${a \in \big( \Stp Q \cap D^*\big) \cdot \St Q}$. This proves \eqref{newclaim},
under the assumption that $T$ is noncommutative.

There remains the case where $T$ is commutative.  
Suppose that $\int(x_j)|_D \ne \id$ for some $j >1$.  Then, we can rearrange
the order of the $x_i$, viewing the order of generators of $\grE$ as 
$x_j, x_2, \ldots, x_{j-1}, x_1, x_{j+1}, \ldots, x_n$  (and their inverses).  
Then, we have $T = D[x_j]$, which is not commutative. So, we 
 are back to the previously proved  case; hence, \eqref{newclaim}  
holds.   

The final possibility is that each $D[x_j]$ is commutative for $j = 1,2, \ldots, n$,
 i.e., $D \subseteq Z(\grE)$.  So, for each $h \in H$,
we have $\tau_h|_D = \int(y_h)|_D \circ \tau|_D = \tau|_D$. 
In the basic argument for \eqref{newclaim}, we used the assumption that 
$T$ is commutative only to be able to apply Lemma~\ref{lem:invDiv} for each
$\tau_h$.  Now $T$ is commutative, but if $\tau|_D \ne \id$ then also each
$\tau_h|_D \ne \id$, so Lemma~\ref{lem:invDiv} still applies for each $\tau_h$;
then the preceding argument for \eqref{newclaim} again goes through.
Thus, we may assume that $\tau|_D = \id$. (Hence, $\grE_0= (Z(\grE^\tau))_0$, 
which implies that $\SK(\grE, \tau) = 1$ by \cite[Th.~4.5]{hazw2}. 
But, we are not finished because we do not know that $\SK(Q, \tau) = 1$.)
For any $y\in \grE^*$, we have $\tau(y) = cy$ for some $c\in \grE_0^* = D^*$,
as each $\grE_\gamma$ is a $\tau$-stable $1$-dimensional $\grE_0$-vector space.
Then $y = \tau(cy) = c^2y$, so $c^2 = 1$.  Thus, $\tau(y) = \pm y$ for each 
$y\in \grE^*$.  If $\tau(x_1) = 1$, then $\tau|_A = \id$, which we saw earlier 
implies $\SK(Q, \tau) = \SK(\grB, \tau) = 1$, which trivially implies \eqref{newclaim}.
Thus, we may assume that $\tau(x_1) = -x_1$.  Likewise, we may assume that 
$\tau(x_j) = -x_j$ for each $j$.  For, if $\tau(x_j) = x_j$, then after reordering the generators of $\grE$
by interchanging $x_1$ and $x_j$,  we have $A = D(x_j)$, so again $\tau|_A = \id$
  and $\SK(Q, \tau) = 1$.  Similarly, we are done if $\tau(x_1x_2) = x_1x_2$. For,
 we can then replace $x_1$ by $x_1 x_2$ as a generator of $\grE$, as 
$\Ga_\grE = \Z(\ga_1+\ga_2) \oplus \Z \ga_2 \oplus \ldots \oplus \Z \ga_n$.  
Then $A = D(x_1x_2)$ and $\tau|_A = \id$, showing $\SK(Q, \tau) = 1$, as before.
Thus, we may assume that $\tau(x_1x_2) = -x_1x_2$.  Then, 
$x_2x_1 = (-1)^2 \tau(x_1 x_2) = - x_1 x_2$.  Likewise, by reordering the $x_i$,
we are done if $\tau(x_i x_j) = x_i x_j$ for any distinct $i$ and $j$.  So, 
we may assume that $\tau(x_i x_j) = -x_i x_j$, and hence $x_j x_i = - x_i x_j$, 
whenever $i \ne j$.  If~$n \ge 3$, we then find that $\tau(x_1x_2x_3) = x_1 x_2 x_3$.  
In this case, we can replace $x_1$ by $x_1x_2 x_3$ as a generator of $\grE$, 
as $\Ga_\grE = \Z (\ga_1 + \ga_2 + \ga_3) \oplus \Z \ga_2 \oplus \ldots \oplus \Z\ga_n$.
Then, $A = D(x_1x_2 x_3)$, and $\tau|_A = \id$, so $\SK(Q,\tau) = 1$, as before.  
The remaining case is that $n = 2$, $D$ is central in $\grE$, $\tau|_D = \id$, $\tau(x_i) = -x_i$
for $i = 1,2$, and $x_2 x_1 = - x_1 x_2$.  But then, $Z(\grE) = D[x_1^2,x_2^2]$ and 
$\tau|_{Z(\grE)} = \id$, which is a contradiction as $\tau$ is unitary.  Thus,
in all cases that can occur \eqref{newclaim} holds, so also   \eqref{claim}. 
Hence, $\SK(Q,\tau) = \SK(\grE, \tau)$, as desired.
\end{proof}
 
Here is a quick consequence of the Theorem:

\begin{corollary}[Stability Theorem \cite{y1var}]  Let $D$ be a division algebra 
finite-dimensional over $Z(D)$, and let $\tau$ be a unitary involution on $D$.  
Then, $\SK(D, \tau) \cong \SK(D(x),\tau')$, where  $D(x)$ is the rational division
algebra over $D$ and $\tau'$ is the canonical extension of $\tau$
to $D(x)$, with $\tau'(x) = x$. 
\end{corollary}

\begin{proof} 
Let $\grE = D[x, x\inv]$, the (untwisted) Laurent polynomial ring in one variable 
over $D$.  Then, $\grE$~is a graded division algebra with $\grE_0 = D$, 
$\Ga_\grE = \Z$, and $\grE_j = Dx^j$ for all $j \in \Z$.  Also, ${q(\grE) = D(x)}$
and ${Z(\grE) = Z(D)[x,x\inv]}$, so $\DIM\grE{Z(\grE)} = \DIM D {Z(D)} < \infty$.
The extension $\tau'$ of $\tau$ to $D(x)$ clearly restricts to a unitary graded 
involution on $\grE$.   Note that $\grE^* = \bigcup_{j\in \Z} D^* x^j$.  For 
$d\in D^*$ and $j \in \Z$, we have ${\tau'(dx^j) = \tau(d)x^j}$ and $\Nrd_\grE(dx^j)
= \Nrd_D(d)x^{mj}$, where $m = \ind(D(x)) = \ind(D)$.  Hence, ${\Sigma_{\tau'}'(\grE)
= \bigcup_{j\in \Z}\Stp D \,x^j}$ and $\Sigma_{\tau'}(\grE) = 
\bigcup_{j\in \Z}\St D\,x^j$, showing that $\SK(\grE, \tau') \cong \SK(D,\tau)$.
(This is a special case of the result in \cite[Cor.~4.10]{hazw2} that if $\grE$ is 
a graded 
division algebra with unitary graded involution $\tau$, and 
$\Ga_\grE = \Ga_{Z(\grE)^\tau}$, 
then $\SK(\grE, \tau) \cong \SK\big(\grE_0, \tau|_{\grE_0}\big)$.)  
Since
${\SK(\grE, \tau') \cong \SK(D(x), \tau')}$ by~Th.~\ref{th:main}, we have
$\SK(D(x), \tau') \cong \SK(D, \tau)$. 
\end{proof}
  
\bigskip 
\vfill\eject   

\section{Examples}

Here are a few examples of $\SK(Q, \tau)$ that follow from known results about 
$\SK(\grE, \tau)$.

For a field $K$ containing a primitive $n$-th root of unity $\omega$ ($n\ge 2$)
and any $a, b\in K^*$, let $\big(\frac{a,b}K\big)_\omega$
denote the degree $n$ symbol algebra over $K$ with generators $i, j$ and relations
$i^n = a$, $j^n = b$, and $ij = \omega ji$.  Note that if $K$ has a nonidentity automorphism 
$\eta$ such that $\eta^2 = \id$ and $\eta(\omega) = \omega\inv$ and if $\eta(a) = a$
and $\eta(b) = b$, then 
$\big(\frac{a,b}K\big)_\omega$ has a unitary involution $\tau$ satisfying 
$\tau(c i^k j ^\ell) = j^\ell i^k \eta (c)$ for all $c\in K$, $k, \ell \in \Z$.   
 
\begin{example} Let $r_1, \ldots, r_m$ be integers with each $r_i \ge 2$, and let 
$s = \lcm(r_1, \ldots, r_m)$ and $n = r_1 \ldots r_m$.
Let $L$ be a field containing a primitive $s$-th root of unity $\omega$, and suppose 
$L$ has an automorphism $\eta$ of order $2$ such that $\eta(\omega) = \omega\inv$.
(For example, take $L = \C$ and $\eta$ to be complex conjugation.)
Let $L = K(x_1, \ldots, x_{2m})$, a rational function field over $L$.
For  $k = 1,2, \ldots, m$, let $\omega_k = \omega^{s/r_k}$, which is a primitive $r_k$-th
root of unity in $L$.  Let 
$$
\textstyle
Q \ = \ \big(\frac{x_1,x_2}K\big)_{\omega_1} \otimes _K \ldots \otimes_K
\big(\frac{x_{2m-1},x_{2m}}K\big)_{\omega_m}.
$$
So, $Q$ is a division algebra over $K$ of exponent $s$ and index $n$.
Extend $\eta$ to an automorphism of order~$2$ of~$K$  by setting $\eta(x_\ell)
= x_\ell$ for $\ell = 1, 2, \ldots, 2m$. Each symbol algebra 
$\big(\frac{x_{2k-1},x_{2k}}K\big)_{\omega_k}$ has a unitary involution $\tau_k$ 
as described above, with $\tau_k |_K = \eta$.  Let 
$\tau = \tau_1 \otimes \ldots \otimes \tau_m\colon Q \to Q$.  Since the $\tau_k$
all agree on $K$, this $\tau$ is a well-defined unitary involution on $Q$.
Let $\mu_\ell(L)$ denote the group of all $\ell$-th roots of unity in $L$.
Then, 
\begin{align}\label{TR}
\begin{split}
\SK(Q, \tau) \, &\cong 
\,\{c\in L^*\mid \eta(c^n) = c^n\} \,\big/\, 
\{ c\in L^*\mid\eta(c^s) = c^s\}   
\\
&\cong \,\{\xi \in \mu_n(L)\mid \eta(\xi) = \xi\inv\} \big/\mu_s(L).
\end{split}
\end{align}
For, let $\grZ = L[x_1, x_1\inv, \ldots, x_{2m}, x_{2m}\inv]$, the (commutative)
iterated Laurent polynomial ring over $L$, with its usual multigrading in  
$x_1, \ldots, x_{2m}$;
that is, $\Ga_\grZ = \Z^{2m}$ with 
$\grZ_{(k_1,\ldots ,k_{2m})} = Lx_1^{k_1}\ldots x_{2m}^{k_{2m}}$. 
 Then, $\grZ$~is a graded field.  Let $\grE$ be the tensor product of graded
symbol algebras, ${\grE = \big(\frac{x_1,x_2}\grZ\big)_{\omega_1} \otimes _\grZ 
\ldots \otimes_\grZ
\big(\frac{x_{2m-1},x_{2m}}\grZ\big)_{\omega_m}}$.  
Then, the grading on $\grZ$ extends to a grading on $\grE$, with 
$\deg(i_k) = \frac 1{r_k}\deg(x_{2k-1})$ and ${\deg(j_k) = 
\frac 1{r_k}\deg(x_{2k})}$, where $i_k$ and $j_k$ are the standard generators
for $\big(\frac{x_{2k-1},x_{2k}}\grZ\big)$.
Then, $\grE$ is a graded division algebra with center $\grZ$,
 and $\grE$ is totally ramified over $\grZ$, i.e., 
$\grE_0 = L = \grZ_0$. Clearly $q(\grZ) = K$ and  $q(\grE) = \grE 
\otimes_\grZ q(\grZ) = Q$,
 and $\tau$ restricts to a unitary graded  involution on $\grE$. Thus,
 $\SK(Q, \tau)\cong \SK(\grE, \tau)$ by Th.~\ref{th:main} , and  formula \eqref{TR} follows from 
the corresponding formula for $\SK(\grE, \tau)$ given in \cite[Th. 5.1]{hazw2},
as $\grE$ is totally ramified over $\grZ$ (i.e., $\grE_0 = \grZ_0$) and $\grZ$
is unramified over $\grZ^\tau$. 
\end{example}

Let $F \subseteq K \subseteq N$ be fields with $N$ Galois over $K$, 
$K$ Galois over $F$, and $\DIM KF = 2$.  Let $\Br(K)$ denote the 
Brauer group of $K$, and $\Br(N/K)= \ker\big(\Br(K) \to \Br(N)\big)$, the relative
Brauer group.  Let $\Br(N/K;F) = \{ [A] \in \Br(N/K) \mid \cor_{K \to F}[A] = 1\}$, 
where $\cor_{K \to F}$ is the corestriction mapping $\Br(K)$ to  $\Br(F)$.  
Recall that the theorem of Albert-Riehm says that $\Br(N/K;F)$ consists of the 
classes of central simple $K$ -algebras $A$ such that $N$ splits $A$ and 
$A$ has a unitary involution $\tau$ such that $K^\tau = F$ (see \cite[Th.~3.1, p.~ 31]{BoI}).
Suppose $N$ is a cyclic Galois extension of $K$ with 
$\DIM NK = n$ and $\Gal(N/K) = \langle 
\sigma\rangle$.  For $b\in K^*$, let $(N/K,\sigma,b)$ denote the 
cyclic algebra 
$$
(N/K, \sigma, b) \, = \, \textstyle \bigoplus\limits_{i = 0}^{n-1}
Ny^i, \quad \text{where \ $yc = \sigma(c) y$  \,for all $c \in K$  \ and \  $y^n = b$}. 
$$
Let $\eta$ be the nonidentity $F$-automorphism of $K$.  Suppose further that 
$N$ is Galois over $F$ and 
$N/F$ is {\it dihedral}. That is, suppose  there
exists $\rho \in \Gal(N/F)$ with $\rho|_K = \eta$, $\rho^2 = \id_N$, and 
$\rho \sigma \rho \inv = \sigma\inv$.  So, $\Gal(N/F) = \langle\sigma, \rho\rangle$,
and this group is dihedral if $n \ge 3$.  Observe that when $N/F$ is dihedral, if 
$b \in F^*$, then $(N/K,\sigma, b)$ has a unitary involution $\tau$ given by 
$\tau(cy^i) = y^i\rho(c)$ for all $c\in K$, $i\in \Z$.  Note that 
$\tau|_K = \eta$, so $K^\tau = F$ and $[(N/K, \sigma, b)] \in \Br(N/K;F)$.

\begin{example}
Let $F \subseteq L$ be fields with $\DIM LF = 2$ and $L$ Galois over $F$ with $\Gal(L/F)
= \{ \id_L, \eta\}$.  Let $N_1$ and $N_2$ be cyclic Galois extensions of $L$
which are linearly disjoint over $L$ with each $N_i$ dihedral over $F$ as just
described.  Let $n_j = \DIM{N_j}L$ and let  $\Gal(N_j/L) = \langle \sigma_j\rangle$; 
 extend each $\sigma_j$ 
to $N_1 N_2$ so that $\sigma_1|_{N_2} = \id_{N_2}$ and $\sigma_2|_{N_1} = \id_{N_1}$. 
Let $K = L(x_1, x_2)$, a rational function field over $L$, and extend
$\eta$ to $K$ by $\eta(x_i) = x_i$ for $i = 1, 2$ and likewise extend the 
$\sigma_j$ to $N_1N_2K$ by $\sigma_j(x_i) = x_i$ for $j = 1,2$, $i = 1,2$.
Let 
$$
Q\, = \, (N_1K/K, \sigma_1, x_1) \otimes_K (N_2K/K, \sigma_2, x_2),
$$
which is a division algebra over $K$ with exponent $\lcm(n_1, n_2)$ and degree
$n_1n_2$.  As noted above,  each $(N_jK/K, \sigma_j, x_j)$
has a unitary involution $\tau_j$ 
 with $\tau_j|_K = \eta$.  Let $\tau = 
\tau_1 \otimes \tau_2$, which is a unitary involution on $Q$.  Then,
\begin{equation}\label{semiram}
\SK(Q, \tau) \, \cong\, \Br(N_1N_2/L;F) \,\big/ [\Br(N_1/L;F) \cdot\Br(N_2/L;F)].
\end{equation}
It was shown in \cite{yinverse} that the right expression in \eqref{semiram} can be 
made into any finite abelian group by choosing $L$ to be an algebraic number field 
and suitably choosing $N_1$ and $N_2$. To view $Q$ as a ring of quotients,
first take $\grZ = L[x_1,x_1\inv, x_2, x_2\inv] \subseteq K$, so $\grZ$ is a
commutative twice-iterated Laurent polynomial ring over $L$, and we give $\grZ$
its usual grading by multi-degree in $x_1$ and $x_2$, as in the preceding 
example.  Let $\grE = N_1N_2[y_1,y_1\inv, y_2, y_2\inv] \subseteq Q$, where
$y_1$ and $y_2$ are the standard generators of the symbol algebras of~$Q$.  
This $\grE$ is a twisted iterated Laurent polynomial ring with $y_1^{n_1} = x_1$,
$y_2^{n_2} = x_2$, $y_1y_2 = y_2 y_1$ and for all $c\in N_1 N_2$,  
$y_1c = \sigma_1(c) y_1$ and $y_2 c = \sigma_2(c)y_2$.  We extend the grading on  
$\grZ = Z(\grE)$ to $\grE$ by setting $\deg(y_1) = (\frac 1{n_1}, 0)$ and
$\deg(y_2) = (0, \frac 1{n_2})$; so,  ${\Gamma_\grE
=\frac1{n_1}\Z \times \frac1{n_2}\Z}$.
  We can see that $\grE$ is a graded division algebra by noting
that $\grE_0 = N_1N_2$, a field, and each homogeneous component $\grE_\ga$ of $\grE$ is a
 $1$-dimensional $\grE_0$-vector space containing a unit of $\grE$.
 This $\grE$~is semiramified
since ${\DIM{\grE_0}{\grZ_0} = \DIM{N_1N_2}{L}= n_1n_2 = \IND{\Ga_\grE}{\Ga_\grZ}}$;
indeed, it is decomposably semiramified in the terminology of \cite{W}, since
$\grE = (N_1\grZ/\grZ, \sigma_1, x_1)\otimes_\grZ (N_2\grZ/\grZ,\sigma_2,x_2)$
which expreses $\grE$ as a tensor product of semiramified graded 
cyclic algebras. Since $\tau$ on $Q$ clearly restricts to a unitary graded involution 
on~ $\grE$ (recall that $\tau(y_1) = y_1$ and $\tau(y_2) = y_2$), Th.~\eqref{th:main}
shows that $\SK(Q, \tau) \cong \SK(\grE, \tau)$.  But further, 
let ${K' = L((x_1))((x_2))}$, a twice-iterated Laurent power series field over $L$, 
and let $D = (N_1K'/K', \sigma_1, x_1) \otimes _{K'} (N_2K'/K', \sigma_2, x_2)$, 
which is a central simple division algebra over $K'$.  Then, the standard
rank $2$ Henselian valuation $v$ on $K'$ has associated graded ring
$\gr(K') = \grZ$, and for the unique extension of $v$ to $D$ we have
$\gr(D) = \grE$.  Because each $N_jK'$ is dihedral over $F((x_1))((x_2))$,
there is a unitary involution $\widehat \tau$ on $D$ built just as for 
$\tau$ on $Q$.  This $\widehat \tau$ is compatible with the valuation
on $D$, and the involution on $\grE$ induced by $\widehat \tau$ is clearly~
$\tau$.  Thus, $\SK(E, \tau) \cong \SK(D, \widehat \tau)$ by~  \eqref{assocgraded} 
above.  But $\SK(D, \widehat \tau)$ was computed in \cite[Th.~5.6]{y}
(with another proof given in \cite[Th.~7.1(ii)]{W}), and the 
formula given there combined with the isomorphisms stated here yield \eqref{semiram}.

\end{example}

\end{document}